\documentclass[11pt]{article}
\usepackage{amsmath,amssymb,amsthm}

\newtheorem{theorem}{Theorem}[section]
\newtheorem{proposition}[theorem]{Proposition}
\newtheorem{lemma}[theorem]{Lemma}
\newtheorem{corollary}[theorem]{Corollary}
\theoremstyle{definition}
\newtheorem{definition}[theorem]{Definition}

\theoremstyle{remark}

\newcommand{\C}{\mathbb{C}}
\newcommand{\K}{\mathbb{K}}

\newcommand{\N}{\mathbb{N}}

\newcommand{\R}{\mathbb{R}}
\newcommand{\Z}{\mathbb{Z}}

\newcommand{\cA}{\mathcal{A}}
\newcommand{\cB}{\mathcal{B}}

\newcommand{\cM}{\mathcal{M}}

\newcommand{\cU}{\mathcal{U}}

\newcommand{\tv}{{\widetilde{v}}}
\newcommand{\tw}{{\widetilde{w}}}
\newcommand{\tsigma}{{\widetilde{\sigma}}}
\newcommand{\trho}{{\widetilde{\rho}}}
\newcommand{\T}{\mathbb{T}}

\newcommand{\vartau}{{\overline{\tau}}}
\newcommand{\varr}{{\overline{r}}}
\newcommand{\Zjs}{\mathcal{Z}}

\DeclareMathOperator{\id}{id}

\DeclareMathOperator{\Aut}{Aut}
\DeclareMathOperator{\Ad}{Ad}
\DeclareMathOperator{\dist}{dist}
\DeclareMathOperator{\tr}{tr}

\begin{document}
\title{Decomposition rank of UHF-absorbing $\mathrm{C}^*$-algebras}
\author{Hiroki Matui \\
Graduate School of Science \\
Chiba University \\
Inage-ku, Chiba 263-8522, Japan 
\and 
Yasuhiko Sato \\
Graduate School of Science \\
Kyoto University \\
Sakyo-ku, Kyoto 606-8502, Japan}
\date{}

\maketitle

\begin{abstract}   
Let $A$ be a unital separable simple $\mathrm{C}^*$-algebra 
with a unique tracial state. 
We prove that if $A$ is nuclear and quasidiagonal, 
then $A$ tensored with the universal UHF-algebra has 
decomposition rank at most one. 
Then it is proved that 
$A$ is nuclear, quasidiagonal and has strict comparison 
if and only if $A$ has finite decomposition rank. 
For such $A$, we also give a direct proof 
that $A$ tensored with a UHF-algebra has tracial rank zero. 
Applying this characterization, we obtain a counter-example to the
Powers-Sakai conjecture.

\end{abstract}

\section{Introduction}\label{Sec1}

E. Kirchberg and W. Winter \cite{KW} introduced 
the notion of decomposition rank for nuclear $\mathrm{C}^*$-algebras 
as a generalization of topological dimension. 
Among others, 
they showed that finite decomposition rank implies quasidiagonality.
The concept of decomposition rank is 
particularly relevant for Elliott's program 
to classify nuclear $\mathrm{C}^*$-algebras by K-theoretic data. 
Indeed, all nuclear stably finite $\mathrm{C}^*$-algebras 
classified so far by their Elliott invariants 
have in fact finite decomposition rank. 
Also, 
this property is expected to be equivalent to other regularity properties. 
W. Winter \cite{Win1} proved that 
any unital separable simple infinite-dimensional $\mathrm{C}^*$-algebra 
 $A$ with finite decomposition rank is $\mathcal{Z}$-absorbing, 
i.e. $A$ absorbs the Jiang-Su algebra $\mathcal{Z}$ tensorially. 
M. R\o rdam \cite{Ror1} showed that 
$\mathcal{Z}$-absorption implies strict comparison 
for unital simple exact $\mathrm{C}^*$-algebras. 
In the present paper we provide the converse direction of these results 
under the assumption that the algebra has a unique tracial state. 

\begin{theorem}\label{MainThm}
Let $A$ be a unital separable simple $\mathrm{C}^*$-algebra 
with a unique tracial state. 
If $A$ is nuclear, quasidiagonal and has strict comparison, 
then the decomposition rank of $A$ is at most three. 
\end{theorem}

The notion of strict comparison was first introduced by B. Blackadar \cite{Bl} for projections, which was called FCQ2. Later this definition is extended to positive elements in the following sense. Throughout this paper, for a $\mathrm{C}^*$-algebra $A$ we denote by $T(A)$ the set of tracial states on $A$. We say that $A$ has {\it strict comparison} if for two positive elements $a$, $b$ in $M_k(A)$ with $\lim_{n\to \infty}\tau(a^{1/n}) < \lim_{n\to \infty}\tau(b^{1/n})$ for any $\tau\in T(A)$ there exist $r_n\in M_k(A)$, $n\in\N$ such that $r_n^*br_n\to a$, where $\tau$ is regarded as an unnormalized trace on $M_k (A)$.  The precise definition of decomposition rank is given in Definition \ref{DefDec} and the proof of the main theorem is given in Section \ref{Sec5}. 

In 2008, A. S. Toms and W. Winter conjectured that properties of
 finite decomposition rank, 
 $\mathcal{Z}$-absorption, and  strict comparison are equivalent
for unital separable simple infinite-dimensional finite nuclear $\mathrm{C}^*$-algebras 
(see \cite[Conjecture 0.1]{Win1} for example). 
The theorem above, together with the results in \cite{KW,Ror1,Win1,MS1}, 
gives a partial affirmative answer to this conjecture in the following sense. 

\begin{corollary}
Let $A$ be 
a unital separable simple infinite-dimensional nuclear $\mathrm{C}^*$-algebra 
with a unique tracial state. 
Then the following are equivalent. 
\begin{enumerate}
\item $A$ has finite decomposition rank. 
\item $A$ is quasidiagonal and is $\mathcal{Z}$-absorbing. 
\item $A$ is quasidiagonal and has strict comparison. 
\end{enumerate}
\end{corollary}

The technique that we develop in proving Theorem \ref{MainThm} also allows us to bound other ranks of nuclear $\mathrm{C}^*$-algebras. In Section \ref{SecTAF}, for $A$ as in Theorem 1.1, we show that $A$ tensored with a UHF-algebra has tracial rank zero. This enables us to remove the technical condition from H. Lin's result \cite[Corollary 5.10]{Lin2} for nuclear $\mathrm{C}^*$-algebras, and to provide a partial answer to his conjecture mentioned in the introduction of \cite{Lin2}. 
As an application,
we construct a counter-example to the Powers-Sakai conjecture.

In Section \ref{SecKir}, we show that 
any Kirchberg algebra has nuclear dimension at most three. We remark that W. Winter and J. Zacharias \cite[Theorem 7.5]{WZ10Adv} proved that if $A$ is a Kirchberg algebra satisfying the Universal Coefficient Theorem (UCT), then it has finite nuclear dimension.

\section{Murray-von Neumann equivalence}\label{Sec2}
A celebrated result of A. Connes says that 
injective factors with separable predual are 
approximately finite dimensional (AFD), see \cite{Con3}. 
The proof relies on his deep study of automorphisms of factors \cite{Con1, Con2}. 
By now, two alternative proofs which do not use automorphisms are known. 
One is due to U. Haagerup \cite{Haa}. 
His proof uses trace preserving completely positive maps 
instead of automorphisms. 
S. Popa gave another short proof by using excision of amenable traces \cite{Pop1}. 
In the present paper, we focus on the proofs by Connes and Haagerup, 
in which the Murray-von Neumann equivalence for projections 
plays an essential role. 
They constructed approximation by finite dimensional subalgebras 
by using partial isometries 
inducing the Murray-von Neumann equivalence for certain projections. 
In this section 
we establish a generalization of the Murray-von Neumann equivalence 
(Lemma \ref{FunLem}). 
In our setting, the `equivalence' is given by two contractions. 
These two contractions will become involved in 
the estimate of decomposition rank 
of UHF-absorbing $\mathrm{C}^*$-algebras in Section \ref{Sec4}.

We say that a $\mathrm{C}^*$-algebra $A$ has {\it strict comparison for projections} if for any two projections $p$ and $q$ in $M_k(A)$ satisfying $\tau(p)< \tau(q)$ for any tracial state $\tau$ of $A$, there exists a partial isometry $v$ in $M_k(A)$ such that $v^*v=p$ and $vv^*\leq q$.
\begin{lemma}\label{FunLem}
Let $A$ be a unital $\mathrm{C}^*$-algebra with strict comparison for projections. 
Let $p,q$ be projections in $A$ and $n\in \N$. 
If $\tau(p)=\tau(q)$ for any tracial state $\tau$ of $A$, 
then there exist two contractions $v_i$ in $A\otimes M_n$, $i=0,1$ such that 
\[
\left\lVert\sum_{i=0,1}v_i^*v_i-p\otimes1_n\right\rVert\leq\frac{4}{n},\quad 
\left\lVert\sum_{i=0,1}v_iv_i^*-q\otimes1_n\right\rVert\leq\frac{4}{n}, 
\] 
and 
\[
\dist(v_i^*v_i,\{p\otimes x\mid x\in M_n\})\leq\frac{2}{n}\quad\text{for}\ i=0,1. 
\] 
\end{lemma}
\begin{proof}
Let $e_j$, $j=1,2,\dots,n$ be mutually orthogonal minimal projections of $M_n$. When $n=1$, $2$, or $3$ the statement is trivial, thus we may assume that $n\geq 4$. 
Since $A$ has strict comparison for projections, 
there exists a partial isometry $r_1\in A\otimes M_n$ such that 
\begin{align*} 
 &r_1^*r_1=p\otimes e_1,& &r_1r_1^* \leq q\otimes (e_1+e_2). \\
\intertext{By the same reason, we obtain a partial isometry $r_2\in A\otimes M_n$ such that} 
 &r_2^*r_2 \leq p\otimes (e_2+e_3),& &r_2r_2^* = q\otimes(e_1+ e_2)- r_1r_1^*. \\
\intertext{In the same way, there exists a partial isometry $r_3\in A\otimes M_n$ such that}
 &r_3^*r_3=p\otimes (e_2+e_3)- r_2^*r_2,& &r_3r_3^* \leq q\otimes (e_3+e_4). 
\end{align*}
Repeating this argument we have partial isometries $r_j\in A\otimes M_n$, $j=1,2,\dots,n-1$. 
Set 
\[ f=\sum_{j=1}^n \frac{n-j}{n}e_j,\quad v_0= \sum_{j=1}^{n-1} \sqrt\frac{n-j}{n}r_j. \]
Since $\{r_jr_j^*|j=1,2,\dots,n-1\}$ and $\{r_j^*r_j|j=1,2,\dots,n-1\}$ are 
sets of mutually orthogonal projections, it follows that 
\[
\lVert v_0^*v_0-p\otimes f\rVert
=\left\lVert\sum_{j=1}^{n-1}\frac{n-j}{n}r_j^*r_j-p\otimes f\right\rVert
\leq\frac{2}{n}
\]
and 
\[
\lVert v_0v_0^*-q\otimes f\rVert
=\left\lVert\sum_{j=1}^{n-1}\frac{n-j}{n}r_jr_j^*-q\otimes f\right\rVert
\leq\frac{2}{n}. 
\]
Similarly we can construct a contraction $v_1\in A\otimes M_n$ such that 
\[
\left\lVert v_1^*v_1-p\otimes\left(1_n-f\right)\right\rVert\leq\frac{2}{n}
\]
and 
\[
\left\lVert v_1v_1^*-q\otimes\left(1_n-f\right)\right\rVert\leq\frac{2}{n}, 
\]
thereby completing the proof. 
\end{proof}
\section{Relative Property (SI)}\label{Sec3}
In our previous work \cite[Proposition 4.8]{MS2}, we showed  certain 
strict comparison of central sequence algebras by using a technical condition named property (SI). The aim of this section is to introduce a relative version of property (SI) and to show a variant of \cite[Proposition 4.8]{MS2}.

Until the end of this section, we let $A$ and $B$ be unital separable simple nuclear infinite-dimensional $\mathrm{C}^*$-algebras and suppose that $B$ has strict comparison. First, we recall the definitions of central sequence algebras. 
Let $\omega$ be a free ultrafilter on $\N$, and let 
\begin{align*}
&\|b\|_2 = \max_{\tau\in T(B)} \tau(b^*b)^{1/2}\quad {\rm for }\ b\in B,\\ 
&c_{\omega}(B)
=\left\{ (b_n)_n\in \ell^{\infty}(\N, B)\ |\ \lim_{n\to\omega}\|b_n\|=0\right\}. 
\end{align*}
We denote by $B^{\omega}$ 
the ultraproduct $\mathrm{C}^*$-algebra $\ell^{\infty}(\N,B)/c_{\omega}(B)$. 

Set 
\[
J=\left\{(b_n)_n \in \ell^{\infty} (\N, B)\
|\ \lim_{n\to\omega} \| b_n\|_2 =0\right\}
\]
and $M=\ell^{\infty} (\N, B)/J$. 
Since $c_{\omega}(B)\subset J$ 
we regard $M$ as a quotient algebra of $B^{\omega}$, 
and denote by $\Phi$ the quotient map from $B^{\omega}$ to $M$. 
Note that if $B$ has finitely many extreme traces 
then $M$ is isomorphic to a finite direct sum of II$_1$-factors.


Let $\pi$ be a unital embedding of $A$ into $B^{\omega}$. We consider the relative commutants  
\[ \cB =  B^{\omega}\cap\pi(A)'\quad{\rm and}\quad \cM =  M\cap\Phi(\pi(A))'.\]
The Choi-Effros lifting theorem allows us to get a sequence of unital completely positive maps $\pi_n : A\rightarrow B$ such that $(\pi_n(a))_n = \pi(a)$ in $B^{\omega}$ for any $a\in A$.

The following theorem is based on the technique of Kirchberg and R\o rdam \cite[Theorem 3.3]{KR}. Combining this result and a generalized property (SI) (Lemma \ref{LemSI}), we shall show the strict comparison of $\cB$ for projections in Proposition \ref{PropStriComp}.
\begin{theorem}[Kirchberg-R\o rdam]\label{ThmKR}
The restriction $\Phi |_{\cB}: \cB \rightarrow \cM$ is surjective.
\end{theorem}
\begin{proof}
Let $x$ be a contraction in $\cM$, and $(b_n)_n\in B^{\omega}$ be such that $\Phi((b_n)_n) = x$. Set 
\[D= C^*( \pi( A) \cup \{(b_n)_n\})\subset B^{\omega}\quad{\rm and} \quad I =\ker (\Phi |_D) \subset D.\] 
Since $x\in \Phi( \pi (A)) '$ it follows that $[\pi (a), (b_n)_n]\in I$ for any $a\in A$. Let $e_{\lambda}\in I$, $\lambda\in \Lambda$ be a quasicentral approximate unit of $I$. Then we have
\begin{align*}
0&= \lim_{\lambda\to \infty} \left\lVert (1-e_{\lambda})[ (b_n)_n, \pi (a)] (1-e_{\lambda})\right\rVert \\
&=\lim_{\lambda\to \infty} \left\lVert[ (1-e_{\lambda})(b_n)_n (1-e_{\lambda}), \pi(a)]\right\rVert \quad{\rm for\ any\ }a\in A.
\end{align*}
This implies that for any finite subset $F$ of $A$ and $\varepsilon >0$ there exists a sequence $e_{F,\varepsilon,n}$, $n\in\N$ of positive contractions in $B$ such that
\[ \lim_{n\to \omega} \left\lVert e_{F, \varepsilon, n}\right\rVert_2=0,\quad \lim_{n\to \omega} \left\lVert[ (1-e_{F, \varepsilon, n}) b_n (1- e_{F, \varepsilon, n}), \pi_n(a)]\right\rVert < \varepsilon,\]
for all $a\in F$.

Let $F_n$, $n\in\N$ be an increasing sequence of finite subsets of $A$ whose union is dense in $A$ and $\varepsilon_n>0$, $n\in\N$ a decreasing sequence which converges to 0. For any $n\in\N$ we inductively obtain a subset $\Omega_n\in \omega$ such that $\Omega_{n+1}\subset \Omega_n$, $\bigcap_{n\in\N} \Omega_n =\emptyset$ and 
\[ \left\lVert e_{F_n, \varepsilon_n, m}\right\rVert_2 < \varepsilon_n,\quad \left\lVert  (1- e_{F_n, \varepsilon_n, m}) b_m (1- e_{F_n, \varepsilon_n, m}), \pi_m (a) ] \right\rVert < \varepsilon_n,\]
for any $m \in \Omega_n$ and $a\in F_n$. We define $e_m =0$ for $m\in \N \setminus\Omega_1$ and $e_m = e_{F_n, \varepsilon_n, m}$ for $m\in \Omega_n\setminus \Omega_{n+1}$. Then it follows that 
\[ \lim_{n\to \omega} \left\lVert e_n\right \rVert_2=0,\quad \lim_{n\to\omega}\left\lVert [ (1-e_n) b_n(1-e_n), \pi_n (a)]\right\rVert =0, \]
for any $\displaystyle a\in \bigcup_{n\in\N}F_n$. This means that
\[ \Phi(((1-e_n)b_n(1-e_n))_n) =x,\quad ((1-e_n)b_n(1-e_n))_n\in \cB.\]
\end{proof}

\begin{lemma}\label{LemSI}
Suppose that $T(B)$ has a compact extreme boundary with finite topological dimension. Then $B$ has property (SI) relative to $\pi(A)$ in the following sense: For any sequences $(e_n)_n$ and $(f_n)_n$ of positive contractions in $B$ satisfying $(e_n)_n$, $(f_n)_n\in \cB$, 
\[ (e_n)_n\in J, \quad \lim_{m\to \infty} \lim_{n\to\omega} \min_{\tau\in T(B)} \tau(f_n^m) >0,\]
there exists a sequence $(s_n)_n$ in $B$  such that $(s_n)_n\in \cB$, 
\[ \lim_{n\to \omega}\left\lVert s_n^*s_n - e_n\right\rVert =0,\quad \lim_{n\to\omega}\left\lVert f_ns_n- s_n\right\rVert =0.\]
\end{lemma}
\begin{proof}
One can prove this lemma in a similar fashion to the proof of \cite[Theorem 1.1]{MS1}. We sketch a proof for the reader's convenience. In the proof we write $x\approx_{\varepsilon} y$ if $\left\lVert x - y\right\rVert < \varepsilon $ for $x$, $y$ in a $\mathrm{C}^*$-algebra.

Let $F$ be a finite subset of contractions in $A$ and $\varepsilon >0$. Since $A$ is nuclear there are two completely positive contractions $\rho: A\rightarrow M_N$ and $\sigma: M_N \rightarrow A$ such that $\sigma\circ\rho (x) \approx_{\varepsilon/2} x$ for $x\in F$. Since $A$ is unital simple infinite-dimensional, applying Voiculescu's theorem to $\rho$ we can obtain $d_i\in A$, $i=1,2,\dots,N$ such that $\rho (x) \approx_{\varepsilon/2} [\varphi(d_i^* x d_j)]_{i,j}$ for some pure state $\varphi$ of $A$ and for all $x\in F$. Identifying $\sigma$ with a positive element of $M_N(A)$ we have $c_{l,i}\in A$, $l,i=1,2,\dots,N$ such that $\displaystyle  \sigma (e_{i,j}) =\sum_{l=1}^N c_{l,i}^*c_{l,j}$ for the standard matrix units $\{ e_{i,j}\}_{i,j}$ of $M_N$. 
Thus it follows that 
\[\sum_{l=1}^N\sum_{i,j=1}^N \varphi(d_i^*xd_j)\pi (c_{l,i}^*c_{l,j}) \approx_{\varepsilon} \pi(x)\quad {\rm for}\ x\in F.\]

Due to the Akemann-Anderson-Pedersen excision theorem, there exists a positive contraction $a\in A$ such that $ad_i^*xd_j a\approx_{\varepsilon/N^2} \varphi(d_i^*xd_j)a^2$ for $x\in F$ and $\varphi(a)=1$. Because $T(B)$ has compact extreme boundary with finite topological dimension, by \cite[Theorem 1.1]{Sat1} (see also \cite{KR, TWW}), 
we obtain a unital embedding of $M_N$ into $M\cap B'$. By Theorem \ref{ThmKR} we obtain a sequence 
of completely positive order zero maps $\psi_n:M_N\to B$ 
such that $\min_{\tau\in T(B)}\tau(\psi_n(1)^m)\to1$ for any $m\in\N$ and 
$\lVert[\psi_n(x),b]\rVert\to0$ as $n\to\omega$ 
for any $x\in M_N$ and $b\in B$. 
Since $A$ is separable, we may further assume 
\[
\lim_{n\to\omega}\lVert[\psi_n(x),\pi_n(y)]\rVert=0,\quad 
\lim_{n\to\omega}\lVert[\psi_n(x),f_n]\rVert=0
\]
for any $x\in M_N$ and $y\in A$. 
Then a similar argument as in the proof of \cite[Lemma 3.4]{MS1} implies
 that there exist sequences $(f_{l,n})_n$, $l=1,2,\dots,N$ of positive contractions in $B$ such that $(f_{l,n})_n\in\cB$, $(f_{l,n}f_n)_n= (f_{l,n})_n$, $(f_{l,n}f_{l',n})_n=0$ for $l\neq l'$ in $B^{\omega}$, and 
\[
 \lim_{n\to\omega} \min_{\tau\in T(B)} \tau (f_{l,n}^m) = \frac{1}{N} \lim_{n\to\omega} \min_{\tau\in T(B)} \tau(f_n^m)\quad{\rm for\ any\ } m\in\N.
\]
Since $A$ is unital and simple, there exist $v_i\in A$, $i=1,2,\dots,k$ such that $\displaystyle \sum_{i=1}^k v_i^* a^2 v_i =1$. Set $ \alpha = \left(\sum_{i=1}^k \left\lVert v_i\right\rVert^2 \right)^{-1} >0$. By $(f_{l,n})_n\in\cB$ we see that for any $m\in\N$
\[\lim_{n\to \omega}\min_{\tau\in T(B)} \tau( f_{l,n}^{m/2} \pi_n (a)^2 f_{l,n}^{m/2} )\geq \alpha \lim_{n\to \omega} \min_{\tau\in T(B)} \tau(f_{l,n}^m) >0.\]
As $\displaystyle\lim_{m\to\infty} \lim_{n\to\omega} \min_{\tau\in T(B)}\tau(f_{l,n}^m) >0$, $\left\lVert a\right\rVert =1$ and $B$ has strict comparison,  there exist sequences $r_{l,n}\in B$, $l=1,2,\dots,N$, $n\in\N$ such that 
\[ \lim_{n\to\omega} \left\lVert r_{l,n}^* f_{l,n}^{1/2} \pi_n (a)^2 f_{l,n}^{1/2} r_{l,n} - e_n\right\rVert =0, \quad \lim_{n\to\omega} \left\lVert r_{l,n}\right\rVert=\lim_{n\to\omega}\left\lVert e_n\right\rVert^{1/2} \leq 1. \]

Define 
\[s_n = \sum_{l=1}^N\sum_{i=1}^N \pi_n (d_i a) f_{l,n}^{1/2} r_{l,n} \pi_n (c_{l,i})\in B, \quad {\rm for\ }n\in\N.\]
Then we have $(f_ns_n)_n = (s_n)_n$ in $B^{\omega}$. And there exists a neighborhood $\Omega\in \omega$ such that for $x\in F$ and $n\in \Omega$ 
\begin{align*}
s_n^* \pi_n(x) s_n &\approx \sum_{l=1}^N\sum_{i,j=1}^N \pi_n (c_{l,i}^*) r_{l,n}^* f_{l,n}^{1/2} \pi_n (ad_i^* xd_j a) f_{l,n}^{1/2} r_{l,n} \pi_n(c_{l,j})\\
&\approx_{\varepsilon} \sum_l \sum_{i,j} \varphi (d_i^*x d_j) \pi_n (c_{l,i}^*)r_{l,n}^* f_{l,n}^{1/2} \pi_n (a^2) f_{l,n}^{1/2} r_{l,n} \pi_n (c_{l,j})\\
&\approx \sum_l \sum_{i,j} \varphi (d_i^*xd_j) \pi_n (c_{l,i}^*c_{l,j}) e_n \approx_{\varepsilon} \pi_n (x) e_n.
\end{align*}

Since $F\subset A$ and $\varepsilon > 0$ are arbitrary,  $A$ is separable, and $\pi$ is unital, we can find another sequence $s_n$, $n\in\N$ of $B$ such that $(f_ns_n)_n=(s_n)_n$, $(s_n^*s_n)_n=(e_n)_n$ in $B^{\omega}$ and $(s_n)_n \in\cB$.
\end{proof}  

\begin{proposition}\label{PropStriComp}
Suppose that $B$ has finitely many extremal tracial states. Then the following holds.
\begin{enumerate}
\item The map $T(\mathcal{M})\to T(\mathcal{B})$ 
induced by $\Phi:\mathcal{B}\to\mathcal{M}$ is surjective. 
\item $\cB$ has strict comparison for projections.
\end{enumerate}
\end{proposition}
\begin{proof}[Proof of {\rm (i)}]
It suffices to show $\tau(x)=0$ 
for any $\tau\in T(\mathcal{B})$ and $x\in\mathcal{B}\cap\ker\Phi$. 
We may assume that $x$ is a positive contraction. 
Let $(x_n)_n$ be a representative sequence of $x$ 
consisting of positive contractions. 
We have 
\[
\lim_{n\to\omega}\min_{\tau\in T(B)}\tau((1-x_n)^m)=1
\]
for any $m\in\N$. 
By Lemma \ref{LemSI}, $B$ has property (SI) relative to $\pi(A)$. 
It follows that there exists $s_1\in\mathcal{B}\cap\ker\Phi$ 
such that $s_1^*s_1=x$ and $(1-x)s_1=s_1$. 
Clearly $x+s_1s_1^*$ is a positive contraction in $\mathcal{B}\cap\ker\Phi$. 
Hence, in the same way we obtain $s_2\in\mathcal{B}\cap\ker\Phi$ 
such that $s_2^*s_2=x$ and $(1-x-s_1s_1^*)s_2=s_2$. 
Then $x+s_1s_1^*+s_2s_2^*$ is again 
a positive contraction in $\mathcal{B}\cap\ker\Phi$. 
Repeating this argument, 
one can find $s_i\in\mathcal{B}\cap\ker\Phi$, $i\in\N$ satisfying 
\[
s_i^*s_i=x\quad\text{and}\quad 
x+s_1s_1^*+s_2s_2^*+\dots+s_ms_m^*\leq1
\]
for every $m\in\N$. 
Hence $\tau(x)=0$ for all $\tau\in T(\mathcal{B})$. 
\end{proof}
\begin{proof}[Proof of {\rm (ii)}]
Let $\pi^{(n)}$ be the amplification map of $\pi$ defined by $\pi^{(n)}(a) = \pi(a)\otimes 1_n\in B^{\omega}\otimes M_n \cong (B\otimes M_n)^{\omega}$, $n\in\N$, $a\in A$. 
By the canonical identification we have 
$\cB\otimes M_n \cong (B\otimes M_n)^{\omega}\cap\pi^{(n)} (A)'$ for any $n\in\N$. Then it suffices to consider two projections in $\cB$.

Assume that two projections $p$ and $q$
in $\cB$ satisfy $\tau(p)< \tau(q)$ for any $\tau\in T(\cB)$. This implies $\vartau\circ\Phi (p) < \vartau\circ\Phi(q)$ for any $\vartau\in T(\cM)$. Because $\cM$ is a finite von Neumann algebra, there exists $\varr\in\cM$ such that $\varr^*\varr = \Phi(p)$ and $\varr\varr^* \leq\Phi(q)$. By Theorem \ref{ThmKR} we obtain $r\in\cB$ such that $\Phi(r)=\varr$ and $\lVert r\rVert\leq 1$. Set $w=qrp\in\cB$ and let $q_n$, $w_n\in B$, $n\in\N$ are representatives of $q$ and $w$. Since $q$ is a projection and $q\geq ww^*$, we may assume that $q_n$, $n\in\N$ are projections and $q_n \geq w_n w_n^*$ for any $n\in\N$. Since $\tau(q-ww^*)\geq \tau(q-p)> 0$ for any $\tau \in T(\cB)$ it follows that 
\[ \lim_{n\to\omega} \tau'(q_n - w_nw_n^*) >0 \quad{\rm for\ any\ }\tau'\in T(B).\] 
Because $\Phi(q-ww^*)^m=\Phi(q-ww^*)$ for any $m\in\N$ and $B$ has finitely many extremal traces, we have
\[ \lim_{n\to\omega}\min_{\tau'\in T(B)} \tau'((q_n-w_nw_n^*)^m) =\lim_{n\to\omega}\min_{\tau'\in T(B)} \tau' (q_n -w_nw_n^*) >0,\]
for any $m\in\N$. By Lemma \ref{LemSI} and $p- w^*w\in \cB\cap\ker\Phi$, there exists $s\in \cB\cap \ker\Phi$ such that $s^*s=p-w^*w$ and $(q-ww^*)s=s$. We define $v=w+s$, then we conclude that
$v^*v=w^*w+ s^*s =p$ and $qv=v$.
 This completes the proof. 
\end{proof}

\section{ Decomposition rank }\label{Sec4}
We denote by $\cU$ the universal UHF-algebra. 
In this section, we show that certain $\cU$-absorbing $\mathrm{C}^*$-algebras have finite decomposition rank. First let us recall the definition of decomposition rank.
\begin{definition}[Kirchberg-Winter \cite{KW,Win1}]\label{DefDec}
A separable $\mathrm{C}^*$-algebra $A$ is said to have {\it decomposition rank at most} $d\in\N\cup\{0\}$ if there exist completely positive contractions $\varphi_n:A\rightarrow \bigoplus_{i=0}^d F_{i,n}$, $n\in\N$ and $\psi_{i,n}: F_{i,n} \rightarrow A$, $i=0,1,\dots,d$, $n\in\N$, where $F_{i,n}$ are finite-dimensional $\mathrm{C}^*$-algebras, such that $\psi_{i,n}$ are order zero (disjointness preserving), $\sum_{i=0}^d\psi_{i,n}: \bigoplus_{i=0}^d F_{i,n}\rightarrow A$ is contractive, and 
\[ \lim_{n\to\infty} \left\lVert \left(\sum_{i=0}^d \psi_{i,n}\right)\circ \varphi_n(a)-a \right\rVert =0\quad{\rm for\ any\ }a\in A,\]
where we simply write $\sum_{i=0}^d\psi_{i,n}(\bigoplus_{i=0}^dx_i)=\sum_{i=0}^d\psi_{i,n}(x_i)$ for $x_i\in F_{i,n}$.
\end{definition}

The proof of the following theorem is based on the crucial technique which was developed by A. Connes and U. Haagerup in the proof of \cite[Theorem 3.1, (d)$\Rightarrow$(a)]{Con3} and \cite[Theorem 4.2]{Haa}. 
\begin{theorem}\label{ThmU}
Let $A$ be a unital separable simple nuclear $\mathrm{C}^*$-algebra with a unique tracial state. If $A$ is quasidiagonal then the decomposition rank of $A\otimes\cU$ is at most one.
\end{theorem}

\begin{lemma}\label{LemDisPre}
Let $A$ be a unital simple stably finite $\mathrm{C}^*$-algebra which absorbs a UHF-algebra tensorially, $A_1$ a $\mathrm{C}^*$-subalgebra of $A$, and $\omega$ a free ultrafilter on $\N$. Suppose that $v_n$, $n\in\N$ are contractions in $A$ such that 
\[ (|v_n|)_n \in A_1^{\omega}\subset A^{\omega}.\]
Then there exists a sequence $(\tilde v_n)_n$ of contractions in $A$ such that 
$|\tilde v_n|\in A_1$ for any $n\in\N$ and 
\[
(\tilde v_n)_n=(v_n)_n\quad {\rm in\ } A^{\omega}.
\]
\end{lemma}
\begin{proof}
Let $\varepsilon_n >0 $, $n\in\N$ be a decreasing sequence which converges to 0, and let $f_n(t) =\max\{ 0, t-\varepsilon_n\}$ for $t\in [0, \infty)$, $n\in\N$. M. R\o rdam showed that any unital simple stably finite UHF-absorbing $\mathrm{C}^*$-algebra has stable rank one \cite[Corollary 6.6]{Ror0}. Then, by G. K. Pedersen's polar decomposition theorem \cite[Theorem 5]{Ped}, there exists a sequence $u_n$, $n\in\N$ of unitaries in $A$ such that
\[ \left\lVert u_n f_n(|v_n|) - v_n \right\rVert \leq \varepsilon_n.\]
 By $(|v_n|)_n \in A_1^{\omega}$ there exists a sequence $h_n$, $n\in\N$ of positive contractions in $A_1$ such that $(h_n)_n = (|v_n|)_n = (f_n(|v_n|))_n$ in $A_1^{\omega}$. 
Define $\tilde v_n=u_nh_n$. 
Then $|\tilde v_n|=h_n\in A_1$ and $(\tilde v_n)_n=(u_nf_n(|v_n|))_n=(v_n)_n$ in $A^\omega$. 
\end{proof}

In what follows, we let $\Ad a(b)=aba^*$ for $a$, $b$ in a $\mathrm{C}^*$-algebra and denote by $\|a\|_{2,\tau}$ the 2-norm of $a$ associated with a tracial state $\tau$, i.e. $\|a\|_{2,\tau}= \tau(a^*a)^{1/2}$.

\begin{proof}[Proof of Theorem \ref{ThmU}]
Let $\iota$ be the canonical unital embedding of $A$ into $A\otimes \cU$. Since $A\otimes M_n$ also satisfies the assumptions in the theorem for any $n\in\N$, it suffices to construct completely positive contractions $\trho_n: A\rightarrow M_{k_n}\oplus M_{k_n}$, $n\in\N$ and $\tsigma_{i,n}: M_{k_n}\rightarrow A\otimes\cU$, $i=0,1$, $n\in\N$ such that $\tsigma_{i,n}$ are order zero, $\sum_{i=0,1}\tsigma_{i,n}: M_{k_n}\oplus M_{k_n} \rightarrow A\otimes \cU$ is contractive, and 
\[
\left\lVert \left(\sum_{i=0,1} \tsigma_{i,n}\right)\circ \trho_n (a) - \iota(a) \right\rVert\to0,\quad n\to \infty,\quad{\rm for\ any\ } a\in A.
\] 

We let $\cU_i$, $i=0,1$ be commuting unital $\mathrm{C}^*$-subalgebras of $\cU$ such that $\cU_0\cong\cU_1\cong\cU$ and $\mathrm{C}^*(\cU_0\cup \cU_1)=\cU$. We denote by $\tau_A$, $\tau_{\cU_i}$, $i=0,1$, and $\tau_{\cU}$ the unique tracial states of $A$, $\cU_i$, and $\cU$. Let $\omega$ be a free ultrafilter on $\N$. 

Since $A$ is quasidiagonal, there are unital completely positive maps $\rho_n: A\rightarrow M_{k_n}$, $n\in\N$ such that 
\[\lim_{n\to \infty} \left\lVert \rho_n (a)\rho_n (b) - \rho_n(ab) \right\rVert =0\quad{\rm and}\quad \lim_{n\to\infty} \left\lVert \rho_n (a)\right\rVert = \left\lVert a \right\rVert,
\] 
for any $a$, $b\in A$ (see \cite[Lemma 7.1.4]{BO} for example). Let $\sigma_n: M_{k_n} \rightarrow \cU_0$, $n\in\N$ be a sequence of unital embeddings. We define unital completely positive maps $\varphi_n : A\rightarrow \cU_0$, $n\in\N$ and $\varphi: A\rightarrow \cU_0 ^{\omega}$ by 
\[\varphi_n(a) = \sigma_n\circ\rho_n(a) \quad{\rm and}\quad \varphi(a)= (\varphi_n(a))_n,\]
for any $n\in\N$ and $a\in A$. Note that $\varphi$ is a unital embedding because $\rho_n$ are asymptotically multiplicative.

In what follows, we identify $A\otimes \cU_0$ with the set of constant sequences in $(A\otimes \cU_0)^{\omega}$ and $\cU_0$ with the unital $\mathrm{C}^*$-subalgebra $1_A\otimes\cU_0$ of $A\otimes\cU_0$. By this canonical unital embedding we also identify $\cU_0^{\omega}$ with the unital $\mathrm{C}^*$-subalgebra of $(A\otimes \cU_0)^{\omega}$. From these identifications it follows that $[\iota(a), \varphi(b)] =0 $ for $a$, $b\in A$. We define the product map $\iota\times \varphi: A\odot A\rightarrow (A\otimes \cU_0)^{\omega}$ by $(\iota\times \varphi) (a\otimes b) =\iota(a)\cdot \varphi(b)$ for $a$, $b\in A$, where $\odot$ means the algebraic tensor product. By the nuclearity of $A$ we can extend $\iota\times \varphi$ to $A\otimes A$. Since $\tau_A$ is the unique tracial state of $A$, $A\otimes A$ also has the unique tracial state $\tau_A\otimes\tau_A$. Then the strong closure of $A\otimes A$ in the GNS-representation, associated with $\tau_A\otimes \tau_A$, is the injective II$_1$-factor, and the flip automorphism on $A\otimes A$ is approximately inner in the strong topology \cite{Con3}, i.e., there exists a sequence $w_n$, $n\in\N$ of unitaries in $A\otimes A$ such that 
\[ 
\lim_{n\to\infty} \left\lVert w_n (a\otimes b) w_n^* -b\otimes a\right\rVert_{2, \tau_A\otimes \tau_A} =0, \]
for any $a$, $b\in A$. 
Set $U_n =(\iota\times \varphi)(w_n)\in (A\otimes \cU_0)^{\omega}$, $n\in\N$. We define the tracial state on $(A\otimes\cU_0)^\omega$ by $(\tau_A\otimes \tau_{\cU_0})_{\omega} ((x_n)_n)= \lim_{n\to\omega} (\tau_A\otimes \tau_{\cU_0}) (x_n)$ for $(x_n)_n \in (A\otimes \cU_0)^{\omega}$. Since $\tau_A\otimes \tau_A$ is a unique tracial state of $A\otimes A$ it follows that $(\tau_A\otimes \tau_{\cU_0})_{\omega}\circ(\iota\times\varphi) (x) =(\tau_A\otimes \tau_A) (x)$ for $x\in A\otimes A$, which implies that 
\[ 
\lim_{n\to \infty} \left\lVert U_n \varphi(a) U_n^* -\iota (a)\right\rVert_{2,(\tau_A\otimes \tau_{\cU_0})_{\omega}} =0,
\] 
for any $a\in A$. Since $A$ is separable, by the standard reindexation trick we obtain a sequence $u_n$, $n\in\N$ of unitaries in $A\otimes \cU_0$ such that 
\[
\lim_{n\to \omega} \left\lVert u_n\varphi_n (a) u_n^* - \iota(a) \right\rVert_{2, \tau_A\otimes \tau_{\cU_0}} =0\quad {\rm for \ any \ }a\in A.
\]

Set 
\begin{align*} 
\cA&= (A\otimes \cU\otimes M_2)^{\omega},\\  
\cA_0 &= (A\otimes \cU_0\otimes M_2)^{\omega} \subset \cA,\\
J_0&= \{ (x_n)_n \in \ell^{\infty} (A\otimes \cU_0\otimes M_2)\ |\ \lim_{n\to\omega} \left\lVert x_n \right\rVert_{2, \tau_A\otimes\tau_{\cU_0}\otimes \tr_2} =0\},\\
M_0 &= \ell^{\infty}(A\otimes \cU_0\otimes M_2)/ J_0,
\end{align*}    
and let $\Phi_0$ be the quotient map from $\cA_0$ to $M_0$. For any $a\in A$, we define
\begin{displaymath}
\pi(a)=\left(\left[
\begin{array}{cc}
\varphi_n(a) & 0 \\ 0 & \iota(a)
\end{array}
\right]\right)_n\in \cA_0,\quad \quad 
u=\left(\left[
\begin{array}{cc}
0 & 0 \\ u_n & 0
\end{array}
\right]\right)_n\in \cA_0.
\end{displaymath}
Note that $\pi$ is a unital embedding of $A$ into $\cA_0$. Because of $\lim_{n\to \omega} \left\lVert u_n\varphi_n(a)- \iota(a) u_n \right\rVert_{2, \tau_A\otimes \tau_{\cU_0}}$ $=0$ for any $a\in A$, we have 
\[ \Phi_0 (u)\in  M_0\cap\Phi_0(\pi(A))'.\]
We also define 
\begin{displaymath}
p=\left[
\begin{array}{cc}
1_{A\otimes\cU_0}& 0 \\ 0 & 0
\end{array}
\right]\in \cA_0,\quad \quad 
q=\left[
\begin{array}{cc}
0 & 0 \\ 0 & 1_{A\otimes\cU_0}
\end{array}
\right]\in \cA_0.
\end{displaymath}
Then $p$ and $q$ belong to $ \cA_0\cap\pi(A)'$. 
Since $u_n$ are unitaries it follows that 
\[ u^*u=p\quad{\rm and}\quad uu^* =q\quad{\rm in }\ \cA_0.\]
As $\Phi_0(u)$, $\Phi_0(p)$, and $\Phi_0(q)$ are in $ M_0\cap\Phi_0(\pi(A))'$ we have 
\begin{align*}
\vartau \circ\Phi_0 (p) &=\vartau \circ \Phi_0(q), \\
\intertext{for any tracial state $\vartau$ of $ M_0\cap\Phi_0(\pi(A))'$. 
Applying Proposition \ref{PropStriComp} (i) to $B=A\otimes\cU_0\otimes M_2$, we have }
\tau(p) &=\tau (q), \\
\end{align*}
for any tracial state $\tau$ of $ \cA_0\cap\pi(A)'$.

By (ii) of Proposition \ref{PropStriComp}, $ \cA_0\cap\pi(A)'$ has strict comparison for projections. Then by Lemma \ref{FunLem}, for each $n\in\N$ we obtain two contractions $\tw_{i,n}$, $i=0,1$ in $(\cA_0\cap\pi(A)')\otimes M_n$ such that 
\[
\left\lVert \sum_{i=0,1} \tw_{i,n}^*\tw_{i,n} - p\otimes 1_n \right\rVert \leq \frac{4}{n}, \quad
\left\lVert \sum_{i=0,1} \tw_{i,n}\tw_{i,n}^* - q\otimes 1_n \right\rVert \leq \frac{4}{n}, 
\]
\[ {\rm and}\quad \dist \left(\tw_{i,n}^*\tw_{i,n}, \left\{ p\otimes x\ |\ x\in M_n\right\}\right)\leq \frac{2}{n} \quad{\rm for\ }i=0,1.\] 
We may think of $M_n$ as a unital subalgebra of $\cU_1$, 
and so for each $n\in\N$ we get 
\[
\left\lVert \sum_{i=0,1} \tw_{i,n}^*\tw_{i,n} - p\right\rVert \leq \frac{4}{n}, \quad
\left\lVert \sum_{i=0,1} \tw_{i,n}\tw_{i,n}^* - q\right\rVert \leq \frac{4}{n}, 
\]
\[ {\rm and}\quad \dist\left(\tw_{i,n}^*\tw_{i,n}, p(1_A\otimes \cU_1 \otimes 1_2)^{\omega}\right)\leq \frac{2}{n} \quad{\rm for\ }i=0,1.\] 
Therefore one can find two sequences $(\tv_{i,n})_n$, $i=0,1$ of contractions in $A\otimes\cU\otimes M_2$ such that $(\tv_{i,n})_n\in  \cA\cap\pi(A)'$ for $i=0,1$,
\[ \left(
\sum_{i=0,1} \tv_{i,n}^* \tv_{i,n}\right)_n =p,\quad \quad \left( \sum_{i=0,1} \tv_{i,n}\tv_{i,n}^*\right)_n =q\quad{\rm in\ } \cA,  
\] 
\[
 {\rm and}\quad (\left\lvert \tv_{i,n}\right\rvert)_n\in p(1_A\otimes \cU_1 \otimes 1_2)^{\omega} \subset p\cA p,\quad{\rm for\ }i=0,1.
\] 
Hence there exist two sequences $(v_{i,n})_n$, $i=0,1$ of contractions in $A\otimes\cU$ such that 
\begin{displaymath}
\left(\left[
\begin{array}{cc}
0& 0 \\ v_{i,n} & 0
\end{array}
\right]\right)_n= (\tv_{i,n})_n \in \cA,\quad i=0,1,
\end{displaymath}
and
\[
 \left(\sum_{i=0,1}v_{i,n}^*v_{i,n} \right)_n =\left( \sum_{i=0,1}v_{i,n}v_{i,n}^*\right)_n = 1_{A\otimes\cU}\quad{\rm in \ } (A\otimes\cU)^{\omega}.
\] 
Because of $(\tv_{i,n})_n\in \cA\cap\pi(A)'$ we have 
\[
(v_{i,n}\varphi_n(a))_n= (\iota(a) v_{i,n})_n\quad {\rm in \ } (A\otimes\cU)^{\omega},
\]
for any $a\in A$. It follows from these conditions that for any $a\in A$
\begin{align*}
\iota(a)
&= \sum_{i=0,1} (v_{i,n}\varphi_n(a) v_{i,n}^*)_n \\
&= \sum_{i=0,1} (\Ad v_{i,n}\circ\sigma_n\circ\rho_n (a))_n\quad{\rm in\ } (A\otimes\cU)^{\omega}. 
\end{align*}

From $\left(\left\lvert \tv_{i,n}\right\rvert\right)_n\in p (1_A\otimes \cU_1\otimes 1_2)^{\omega}$ we have $\left(\left\lvert v_{i,n}\right\rvert\right)_n \in (1_A\otimes \cU_1)^{\omega}\subset (A\otimes \cU)^{\omega}$. By Lemma \ref{LemDisPre} we may assume that $\lvert v_{i,n}\rvert$ is in $1_A\otimes\cU_1$ for every $n\in\N$. 
For $i=0,1$ and $n\in\N$, we define an order zero completely positive contraction $\tsigma_{i,n}: M_{k_n} \rightarrow A\otimes \cU$, $n\in\N$ by $\tsigma_{i,n}= \Ad v_{i,n}\circ\sigma_n$ and simply write $\sum_{i=0,1}\tsigma_{i,n}(x_0\oplus x_1)=\sum_{i=0,1}\tsigma(x_i)$ for $x_i\in M_{k_n}$. Since $\iota$ and $\rho_n$ are unital it follows that 
\[
\lim_{n\to\omega} \left\lVert \sum_{i=0,1}\tsigma_{i,n} \right\rVert =1.
\]
By multiplying $\tsigma_{i,n}$ by suitable scalars, 
we may further assume that $\sum_{i=0,1}\tsigma_{i,n}$ are contractive. 
Define completely positive contractions $\trho_n$, $n\in\N$ from $A$ to $M_{k_n}\oplus M_{k_n}$ by $\trho_n(a) = \rho_n(a)\oplus \rho_n(a)$. Then we conclude that 
\[
\iota(a) =\left(\left( \sum_{i=0,1} \tsigma_{i,n}\right)\circ \trho_n (a)\right)_n\quad {\rm in \ } (A\otimes \cU)^{\omega},
\]
for any $a\in A$. This completes the proof.
\end{proof}

\section{$\mathcal{Z}$-absorbing $\mathrm{C}^*$-algebras}\label{Sec5}
In this section, we prove Theorem \ref{MainThm}. 

\begin{proof}[Proof of Theorem \ref{MainThm}]
By Theorem \cite[Theorem 1.1]{MS1}, 
$A$ is $\mathcal{Z}$-absorbing, i.e. $A\cong A\otimes\mathcal{Z}$. 
Let $F$ be a finite subset of contractions in $A$ and let $\varepsilon >0$. 
By Theorem \ref{ThmU}, decomposition rank of $A\otimes\cU$ is at most one, 
and so there exist completely positive contractions 
$\varphi:A\rightarrow E_0\oplus E_1$ and 
$\psi_i:E_i\rightarrow A\otimes\cU$, $i=0,1$ 
such that $E_i$ are finite dimensional $\mathrm{C}^*$-algebras, 
$\psi_i$ are order zero, $\sum_{i=0,1}\psi_i : E_0\oplus E_1\rightarrow A\otimes\cU$ is contractive and 
\[
\left\lVert\left(\sum_{i=0,1}\psi_i\right)\circ\varphi(a)
-a\otimes 1_{\cU}\right\rVert<\varepsilon/2
\]
for any $a\in F$. 
Any order zero completely positive map from a finite dimensional $\mathrm{C}^*$-algebra 
can be identified with a $*$-homomorphism from the cone over the algebra 
(see \cite[Proposition 1.2]{Win1} for instance). 
A cone over a finite dimensional $\mathrm{C}^*$-algebra is projective \cite{Lor}. 
Therefore we may assume that 
$\psi_i$, $i=0,1$ satisfy $\psi_i(E_i)\subset A\otimes M_l$ 
for some unital matrix subalgebra $M_l$ of $\cU$. 

Choose $m\in\N$ so that $(m{-}1)^{-1}<\varepsilon/2$. 
Let $e_1,e_2,\dots,e_{lm}\in M_{lm}$ and $f_0,f_1,\dots,f_{lm}\in M_{lm+1}$ be 
mutually orthogonal minimal projections. 
For $j=1,2,\dots,m$, we put 
\[
I_j=\{n\in\N\mid(j-1)l+1\leq n\leq jl\}. 
\]
Let $\alpha_j:M_l\to M_{lm}$, $\beta_j:M_l\to M_{lm+1}$ 
and $\gamma_j:M_l\to M_{lm+1}$ be embeddings such that 
\[
\alpha_j(1_l)=\sum_{n\in I_j}e_n,\quad 
\beta_j(1_l)=\sum_{n\in I_j}f_n,\quad 
\gamma_j(1_l)=\sum_{n\in I_j}f_{n-1}. 
\]
Take a unitary $w\in M_l\otimes M_l$ 
such that $w(x\otimes y)w^*=y\otimes x$ for any $x,y\in M_l$. 
Let $u:[0,1]\to M_l\otimes M_l$ be a continuous path of unitaries 
such that $u(0)=1_l$ and $u(1)=w$. 
Define a continuous path of unitaries 
$\tilde u:[0,1]\to M_{lm}\otimes M_{lm+1}$ by 
\[
\tilde u(t)
=1_{lm}\otimes f_0+\sum_{j,k=1}^m(\alpha_j\otimes\beta_k)(u(t)). 
\]
For $i=0,1$, we define order zero completely positive maps 
$\lambda_i:M_l\to C([0,1])\otimes M_{lm}\otimes M_{lm+1}$ by 
\[
\lambda_0(x)(t)=\tilde u(t)\left(
\sum_{j=1}^m(1-t)\alpha_j(x)\otimes f_0+
\sum_{j,k=1}^m\left(1-\frac{(m{-}k)t}{m{-}1}\right)
\alpha_j(x)\otimes\beta_k(1_l)
\right)
\tilde u(t)^*
\]
and 
\[
\lambda_1(x)(t)=\sum_{k=1}^m\frac{(m{-}k)t}{m{-}1}1_{lm}\otimes\gamma_k(x). 
\]
It is straightforward to check 
\[
1-\frac{t}{m{-}1}\leq\lambda_0(1_l)(t)+\lambda_1(1_l)(t)\leq1
\]
for any $t\in[0,1]$. This means that $\lambda_0+\lambda_1$ is contractive.
Moreover one has 
\[
\lambda_0(x)(0)=\sum_{j=1}^m\alpha_j(x)\otimes1_{lm+1},\quad 
\lambda_0(x)(1)=
\sum_{k=1}^m\left(1-\frac{m{-}k}{m{-}1}\right)
1_{lm}\otimes\beta_k(x)
\]
and 
\[
\lambda_1(x)(0)=0,\quad 
\lambda_1(x)(1)=\sum_{k=1}^m\frac{m{-}k}{m{-}1}1_{lm}\otimes\gamma_k(x). 
\]
Hence we may regard $\lambda_i$ as a map 
from $M_l$ to the dimension drop algebra 
\[
Z=\{g\in C([0,1])\otimes M_{lm}\otimes M_{lm+1}\mid 
g(0)\in M_{lm}\otimes\C,\ g(1)\in\C\otimes M_{lm+1}\}. 
\]
Then $(\id_A\otimes\lambda_i)\circ\psi_j$ is an order zero map 
from $E_j$ to $A\otimes Z$ for each $i,j=0,1$. 
We also obtain 
\begin{align*}
&\left\lVert\left(\sum_{i,j}(\id_A\otimes\lambda_i)\circ\psi_j\right)\circ(\varphi\oplus\varphi)(a) -a\otimes1_Z\right\rVert\\
&<\left\lVert\sum_{i=0,1}\id_A\otimes\lambda_i
(a\otimes 1_l)-a\otimes1_Z\right\rVert+\frac{\varepsilon}{2}\\
&\leq\frac{1}{m{-}1}+\frac{\varepsilon}{2}<\varepsilon
\end{align*}
for every $a\in F$. 
Since there exists a unital embedding $Z\to\mathcal{Z}$, 
we can conclude that 
$A\cong A\otimes\mathcal{Z}$ has decomposition rank at most three. 
\end{proof}

\section{$\mathrm{C}^*$-algebras with tracial rank zero}\label{SecTAF}
In \cite{Lin, Lin2}, H. Lin introduced the notion of tracial rank, which is based on ideas of S. Popa \cite{Pop2}.
W. Winter \cite{Win07JFA,Win1} proved that if a unital separable simple $\mathrm{C}^*$-algebra $A$ has (locally) finite decomposition rank and has real rank zero, then $A$ has tracial rank zero. In this section, using the same strategy as in the proof of Theorem \ref{ThmU} we give a direct proof that $A$ has tracial rank zero under the assumption that the $\mathrm{C}^*$-algebra absorbs a UHF-algebra tensorially and has a unique tracial state. 
\begin{theorem}\label{ThmTAF}
Let $A$ be a unital separable simple nuclear quasidiagonal $\mathrm{C}^*$-algebra with a unique tracial state. Then the tracial rank of $A\otimes B$ is zero for any UHF-algebra $B$. 
\end{theorem}
\begin{proof}
Because of \cite[Lemma 2.4]{MS0} it is enough to show that $A\otimes \cU$ has tracial rank zero. 
We use the same notations as in the proof of Theorem \ref{ThmU}. Let $\varepsilon >0$. Since $A\otimes M_n$ also satisfies the same assumption in the theorem for any $n\in\N$, it suffices to construct a sequence of finite dimensional $\mathrm{C}^*$-subalgebras $B_n$, $n\in\N$ of $A\otimes \cU$ such that
\[
\tau_{A\otimes\cU}\left(1_{A\otimes\cU} -1_{B_n}\right)<\varepsilon,\quad  
\lim_{n\to\omega} \left\lVert \left[ 1_{B_n}, \iota(a) \right] \right\rVert =0,
\]
and 
\[
\lim_{n\to\omega}\dist(\iota(a)1_{B_n},B_n)=0, 
\quad{\rm for\ any\ } a \in A. 
\]

Recall that $\varphi_n:A\to\cU_0$ is the composition of 
the unital completely positive map $\rho_n:A\to M_{k_n}$ and 
the unital embedding $\sigma_n:M_{k_n}\to\cU_0$. 
Clearly we have $\left[\varphi_n(a), 1_A\otimes x\right]=0$ for any $n\in\N$, $a\in A$, and $x\in \cU_1$. For any $\tau\in T(\cA\cap\pi(A)')$ we have seen that $\tau(p)=\tau(q)=1/2$ in the proof of Theorem \ref{ThmU} . 
Thus for any $\tau\in T(\cA\cap\pi(A)')$, the functional 
\[
\cU_1\ni x\mapsto2\tau\left(\left(
\begin{bmatrix}
1_A\otimes x& 0 \\ 0 & 0
\end{bmatrix}
\right)_n\right)\in\C
\]
gives the unique tracial state $\tau_{\cU_1}$ of $\cU_1$. 
Let $e_0$ be a projection in $\cU_1$ such that $1-\varepsilon < \tau_{\cU_1}(e_0) <1$. Set 
\begin{displaymath}
e=\left(\left[
\begin{array}{cc}
1_A\otimes e_0& 0 \\ 0 & 0 
\end{array}
\right]\right)_n \in \cA\cap\pi(A)'.
\end{displaymath}
Then we have 
\[
\tau(e)=\frac{1}{2}\tau_{\cU_1}(e_0) < \frac{1}{2}=\tau(p) =\tau (q)\quad {\rm for\ any\ } \tau \in T(\cA\cap\pi(A)').
\]

Applying (ii) of Proposition \ref{PropStriComp} to $B= A\otimes \cU\otimes M_2$, we see that $\cA\cap\pi(A)' $ has strict comparison for projections. We obtain a partial isometry $\tv\in\cA\cap\pi(A)'$ such that 
\[ \tv^*\tv = e\quad{\rm and\ }\quad q \tv =\tv. \]
By the definition of $p$ and $q$ there exists a sequence $v_n$, $n\in\N$ of contractions in $A\otimes\cU$ such that 
\begin{displaymath}
\left(\left[
\begin{array}{cc}
0 & 0 \\ v_n & 0 
\end{array}
\right]\right)_n
= \tv\quad{\rm in\ } \cA,\quad{\rm and}\quad (v_n^*v_n)_n = (1_A\otimes e_0)_n\quad{\rm in \ } (A\otimes \cU)^{\omega}.
\end{displaymath}
By perturbing $v_n$ slightly, we may further assume $v_n^*v_n=1_A\otimes e_0$. 
Because of $\tv\in \pi(A)'$ we have 
\[ 
\left(v_n\varphi_n(a)\right)_n =\left(\iota(a)v_n\right)_n\quad{\rm in\ }(A\otimes\cU)^{\omega}\ {\rm for\ any\ }a\in A.
\] 
Then it follows that $[\left(v_nv_n^*\right)_n, \iota(a) ] =0$ in $(A\otimes \cU)^{\omega}$ for any $a\in A$. 
We define finite dimensional $\mathrm{C}^*$-subalgebras $B_n\subset A\otimes\cU$ 
for $n\in\N$ by 
\[
B_n=\Ad v_n \left(\sigma_n(M_{k_n})\right).
\]
One can verify that these $B_n$ satisfy the desired properties.
\end{proof}

The following corollary is a direct consequence of 
\cite[Theorem 5.4]{LN}, \cite[Theorem 1.1]{MS1} and the theorem above. 
\begin{corollary}\label{CorTAF}
Let $A$ be a unital separable simple nuclear quasidiagonal $\mathrm{C}^*$-algebra with a unique tracial state. Suppose that $A$ has strict comparison and satisfies the UCT. Then $A$ is classifiable by the Elliott invariants and is isomorphic to a unital simple approximately subhomogeneous algebra.
\end{corollary}

In the rest of this section, 
we would like to discuss an application of the corollary above. 
R. T. Powers and S. Sakai \cite{PS75CMP} conjectured that 
every strongly continuous action $\alpha:\R\curvearrowright A$ 
on a UHF algebra $A$ is approximately inner. 
An action $\alpha:\R\curvearrowright A$ is said to be strongly continuous 
if the map $t\mapsto\alpha_t(a)$ is continuous for every $a\in A$. 
An action $\alpha:\R\curvearrowright A$ is said to be approximately inner 
if there exists a sequence $(h_n)_n$ of self-adjoint elements in $A$ such that 
\[
\max_{\lvert t\rvert\leq1}\lVert e^{ith_n}ae^{-ith_n}-\alpha_t(a)\rVert\to0
\]
as $n\to\infty$ for all $a\in A$. 
A. Kishimoto \cite{K03JFA} gave 
a counter-example to the AF version of the Powers-Sakai conjecture. 
Namely, he constructed a unital simple AF algebra $A$ and 
a strongly continuous action $\alpha:\R\curvearrowright A$ 
which is not approximately inner. 
Here we shall present a counter-example to the Powers-Sakai conjecture 
by using our main result. 
In what follows, 
an action of $\T=\R/\Z$ is identified with a periodic action of $\R$. 

\begin{theorem}
Let $A$ be a unital simple infinite-dimensional AF algebra with a unique tracial state $\tau$. 
There exists a strongly continuous action $\alpha:\T\curvearrowright A$ 
such that the crossed product $A\rtimes_\alpha\T$ is simple. 
In particular, $\alpha$ is not approximately inner as an $\R$-action. 
\end{theorem}
\begin{proof}
Let $G=\bigoplus_\Z K_0(A)$ be the infinite direct sum of $K_0(A)$ over $\Z$. 
Define $G^+\subset G$ by 
\[
G^+=\left\{(x_n)_n\in G\mid\sum_{n\in\Z}\tau_*(x_n)>0\right\}\cup\{0\}. 
\]
Since $\tau_*(K_0(A))$ is dence in $\R$, $(G,G^+)$ is a simple dimension group. 
Let $B$ be a simple stable AF algebra 
whose $K_0$-group is isomorphic to $(G,G^+)$. 
The AF algebra $B$ admits a densely defined lower semi-continuous trace, 
unique up to constant multiples. 
Choose a projection $e\in B$ 
whose $K_0$-class is equal to $(\dots,0,0,[1_A],0,0,\dots)\in G$, 
where $[1_A]$ is in the $0$-th summand. 
There exists an automorphism $\sigma\in\Aut(B)$ 
such that $K_0(\sigma)$ is the shift on $G$, 
that is, $K_0(\sigma)((x_n)_n)=(x_{n+1})_n$. 
The trace on $B$ is invariant under $\sigma$.

Let $\gamma\in\Aut(\mathcal{Z})$ be an automorphism of $\mathcal{Z}$ 
which has the weak Rohlin property (\cite[Definition 1.1]{S10JFA}). 
We consider the crossed product 
\[
C=(B\otimes\mathcal{Z}\otimes\mathcal{Z})
\rtimes_{\sigma\otimes\gamma\otimes\id}\Z
=((B\otimes\mathcal{Z})\rtimes_{\sigma\otimes\gamma}\Z)\otimes\mathcal{Z}. 
\]
Clearly $C$ is $\mathcal{Z}$-stable. 
Since the trace on $B\otimes\mathcal{Z}\otimes\mathcal{Z}$ is 
preserved by the automorphism, 
$C$ admits a densely defined lower semi-continuous trace. 
Moreover, this trace is unique up to scalar multiples, 
because $\gamma$ has the weak Rohlin property 
(one can prove this in the same way as \cite[Remark 2.8]{MS0}). 
By \cite{B98JFA}, $C$ is AF embeddable, and hence is quasidiagonal. 
By the Pimsner-Voiculescu exact sequence, 
we have $K_0(C)\cong K_0(A)$ and $K_1(C)=0$. 

Set $D=(e\otimes1\otimes1)C(e\otimes1\otimes1)$. 
Then $D$ is a unital separable simple nuclear $C^*$-algebra, 
which is $\mathcal{Z}$-stable, has a unique tracial state 
and is quasidiagonal. 
Besides, $(K_0(D),K_0(D)^+,[1_D],K_1(D))\cong(K_0(A),K_0(A)^+,[1_A],K_1(A))$. 
It follows from the corollary above that $D$ is isomorphic to $A$. 
Let $\alpha:\T\curvearrowright C$ be 
the dual action of $\sigma\otimes\gamma\otimes\id$. 
The projection $e\otimes1\otimes1$ is invariant under $\alpha$, 
and so we can restrict $\alpha$ to $D$. 
The crossed product $D\rtimes_\alpha\T$ is 
stably isomorphic to $C\rtimes_\alpha\T$, 
which is isomorphic to $B\otimes\mathcal{Z}\otimes\mathcal{Z}\otimes\K$ 
by the duality theorem. 
Therefore $D\rtimes_\alpha\T$ is simple. 
In particular, $\alpha$ on $D$ is not approximately inner 
by \cite[Theorem 2.3]{PS75CMP} and \cite[Remark 3.5]{BEH80CMP} 
\end{proof}

\section{Kirchberg algebras}\label{SecKir}

A separable simple nuclear purely infinite $\mathrm{C}^*$-algebra is called a Kirchberg algebra. 
 In this section we prove the following theorem. See \cite[Definition 2.1]{WZ10Adv} for the definition of nuclear dimension.

\begin{theorem}\label{Kirchberg}
Any Kirchberg algebra has nuclear dimension at most three. 
\end{theorem}
\begin{proof}
By \cite[Corollary 2.8]{WZ10Adv}, 
it suffices to show that a unital Kirchberg algebra $A$ has nuclear dimension at most three. 
Because of the Kirchberg-Phillips theorem, 
$A$ is isomorphic to $A\otimes\mathcal{O}_\infty$. 
We let $B_i$, $i=0,1$ be commuting unital $\mathrm{C}^*$-subalgebras of $\mathcal{O}_\infty$ such that $B_0\cong B_1 \cong \mathcal{O}_\infty$ and $\mathrm{C}^*(B_0\cup B_1)= \mathcal{O}_\infty$. 
Let $s\in B_0$ be a non-unitary isometry and 
put $q=ss^*$, $p=1_{B_0}-q$. 
Define $\iota:A\to A\otimes B_0$ by $\iota(a)=a\otimes q$. 
The Kirchberg-Phillips theorem also tells us that 
there exists a unital embedding $\rho:A\to\mathcal{O}_2$. 
Since $pB_0p$ contains a unital copy of $\mathcal{O}_2$, 
there exists a unital embedding 
$\sigma:\mathcal{O}_2\to1_A\otimes pB_0p\subset A\otimes B_0$. 
Let $\varphi=\sigma\circ\rho$. 

We would like to apply the same argument as the proof of Theorem \ref{ThmU} 
to $\varphi$ and $\iota$. 
Set 
\[
\mathcal{A}=(A\otimes \mathcal{O}_\infty )^\omega,\quad 
\mathcal{A}_0
=(A\otimes B_0)^\omega
\subset\mathcal{A}. 
\]
Define a unital homomorphism $\pi:A\to\mathcal{A}_0$ by 
\[
\pi(a)=\varphi(a)+\iota(a)
\]
for $a\in A$. 
By \cite[Corollary 4]{Kir}, 
$\mathcal{A}_0\cap\pi(A)'$ is purely infinite simple. 
In particular, 
for any non-zero projections $e,f\in(\mathcal{A}_0\cap\pi(A)')\otimes M_k$, 
$e$ is Murray-von Neumann equivalent to a subprojection of $f$. 
Hence, for any $n\in\N$, 
we can apply (the proof of) Lemma \ref{FunLem} and obtain contractions\[
w_{i,n}\in(\mathcal{A}_0\cap\pi(A)')\otimes(M_n\oplus M_{n+1}),\quad i=0,1,
\]
such that 
\[
\left\lVert\sum_{i=0,1}w_{i,n}^*w_{i,n}
-1_A\otimes p\otimes(1_n\oplus 1_{n+1})\right\rVert\leq\frac{4}{n},\quad 
\left\lVert\sum_{i=0,1}w_{i,n}w_{i,n}^*
-1_A\otimes q\otimes(1_n\oplus 1_{n+1})\right\rVert\leq\frac{4}{n},
\]
and 
\[\dist(w_{i,n}^*w_{i,n}, \{1_A\otimes p\otimes x\mid x\in M_n\oplus M_{n+1}\})\leq \frac{2}{n}. 
\]
Since $M_n\oplus M_{n+1}$ can be unitally embedded in $B_1\cong \mathcal{O}_{\infty}$ and $n\in\N$ is arbitrary, there exist $v_i\in \cA\cap \pi(A)'$, $i=0, 1$ such that 
$\sum_{i=0,1} v_i^* v_i= 1_A\otimes p$, $\sum_{i=1,0} v_i v_i^*=1_A\otimes q$, and $v_i^*v_i \in (1_A\otimes p B_1)^{\omega}$, $i=0,1$. 
 
In the same way as the proof of Theorem \ref{ThmU}, one has 
\begin{align*}
a\otimes 1_{\mathcal{O}_\infty}&=(1_A\otimes s^*)\iota(a)(1_A\otimes s)\\
&=(1_A\otimes s^*)\left(\sum_{i=0,1} v_i\varphi(a) v_i^*\right)
(1_A\otimes s)\\
&=\sum_{i=0,1}\Ad(1_A\otimes s^*)\circ\Ad v_i
\circ\sigma\circ\rho(a)\quad{\rm in\ }\mathcal{A},
\end{align*}
for every $a\in A$. 
For $i=0,1$, we define an order zero completely positive contraction 
$\tilde\sigma_i:\mathcal{O}_2\to\mathcal{A}$ 
by $\tilde\sigma_i=\Ad(1_A\otimes s^*)\circ\Ad v_i\circ\sigma$. 

Let $F\subset A$ be a finite subset and let $\varepsilon>0$. 
By \cite[Theorem 7.4]{WZ10Adv}, 
the nuclear dimension of $\mathcal{O}_2$ is one. 
It follows that there exist completely positive contractions 
$\alpha:\mathcal{O}_2\to E_0\oplus E_1$ and $\beta_j:E_j\to \mathcal{O}_2$, $j=0, 1$, 
such that $E_j$ are finite dimensional $\mathrm{C}^*$-algebras, 
$\beta_j$ are order zero and 
\[
\left\lVert\left(\sum_{j=0,1}\beta_j\right)\circ\alpha(\rho(a))
-\rho(a)\right\rVert<\varepsilon
\] 
for any $a\in F$. 
Then $\tilde\sigma_i\circ\beta_j$ is an order zero map 
from $E_j$ to $\mathcal{A}$ for each $i,j=0,1$. Define a completely positive contraction $\tilde\rho :A\rightarrow E_0\oplus E_1\oplus E_0\oplus E_1$ by $\tilde\rho(a)=\alpha(\rho(a))\oplus\alpha(\rho(a))$. 
We have 
\[
\left\lVert\left(\sum_{i,j=0,1}\tilde\sigma_i\circ \beta_j\right)\circ\tilde\rho(a)-a\otimes 1_{\mathcal{O}_\infty}\right\rVert<\varepsilon
\] 
for any $a\in F$. 
Since the cone over $E_j$ is projective, 
$\tilde\sigma_i\circ\beta_j$ lifts to 
an order zero map to $\ell^\infty(\N,A\otimes \mathcal{O}_\infty)$. 
The proof is completed. 
\end{proof}


\begin{thebibliography}{99}
\bibitem{AAP} C. A. Akemann, J. Anderson, and G. K. Pedersen, 
\textit{ Excising states of $\mathrm{C}^*$-algebras}, 
Canad. J. Math. 38 (1986), 1239--1260.
\bibitem{Bl} B. Blackadar, 
\textit{ Comparison theory for simple $\mathrm{C}^*$-algebras},
Operator algebras and applications, London Math. Soc. Lecture Note Ser., 135, Cambridge Univ. Press, Cambridge, (1988), 21--54.
\bibitem{BEH80CMP}
O. Bratteli, G. Elliott and R. Herman, 
\textit{On the possible temperatures of a dynamical system}, 
Comm. Math. Phys. 74 (1980), 281--295. 


\bibitem{B98JFA}
N. Brown, 
\textit{AF embeddability of crossed products of AF algebras by the integers}, 
J. Funct. Anal. 160 (1998), 150--175. 
\bibitem{BO} N. P. Brown and N. Ozawa, 
\textit{ $\mathrm{C}^*$-algebras and finite-dimensional approximations},
 Graduate Studies in Mathematics, 88. Amer. Math. Soc., Providence, RI, 2008.
\bibitem{Con1} A. Connes, 
\textit{Periodic automorphisms of the hyperfinite factor of type II$_1$}, 
Acta Sci. Math. 39 (1977), no. 1-2, 39--66. 
\bibitem{Con2} A. Connes, 
\textit{Outer conjugacy classes of automorphisms of factors},
Ann. Sci. ${\rm \acute{E}}$cole Norm. Sup. (4) 8 (1975), no. 3, 383--419.
\bibitem{Con3} A. Connes, 
\textit{ Classification of injective factors: cases II$_1$, II$_\infty$, III$_{\lambda}$, $\lambda\neq 1$}, 
Ann. of Math. 104 (1976), 73--115.

\bibitem{Haa} U. Haagerup, 
\textit{ A new proof of the equivalence of injectivity and hyperfiniteness for factors on a separable Hilbert space}, 
J. Funct. Anal. 62 (1985), no. 2, 160--201. 
\bibitem{JS} X. Jiang and H. Su, 
\textit{On a simple unital projectionless $\mathrm{C}^*$-algebra}, 
Amer. J. Math. 121 (1999), 359--413.
\bibitem{Kir}
E. Kirchberg, 
\textit{The classification of purely infinite $C^*$-algebras 
using Kasparov's theory}, 
preprint, 1994. 
\bibitem{KR} E. Kirchberg and M. R\o rdam,
\textit{Central sequence $\mathrm{C}^*$-algebras and tensorial absorption of the Jiang-Su algebra},
To appear in J.  Reine Angew. Math.
\bibitem{KW}E. Kirchberg and W. Winter,
\textit{Covering dimension and quasidiagonality},
Internat. J. Math. 15 (2004), no. 1, 63--85.

\bibitem{K03JFA}
A. Kishimoto, 
\textit{Non-commutative shifts and crossed products}, 
J. Funct. Anal. 200 (2003), 281--300. 

\bibitem{Lin} H. Lin, 
\textit{The tracial topological rank of $\mathrm{C}^*$-algebras}, Proc. London Math. Soc. (3) 83 (2001), no. 1, 199--234.
\bibitem{Lin2} H. Lin, 
\textit{Tracially AF $\mathrm{C}^*$-algebras},
Trans. Amer. Math. Soc. 353 (2001), no. 2, 693--722. 

\bibitem{LN}
H. Lin and Z. Niu, 
\textit{Lifting $KK$-elements, asymptotical unitary equivalence and 
classification of simple $\mathrm{C}^*$-algebras}, 
Adv. Math. 219 (2008), 1729--1769. 

\bibitem{Lor} 
T. Loring,
\textit{Lifting solutions to perturbing problems in $\mathrm{C}^*$-algebras},
Fields Institute monographs
8, AMS, Providence, Rhode Island (1997).

\bibitem{MS0} H. Matui and Y. Sato,
\textit{$\Zjs$-stability of crossed products by strongly outer actions},
Comm. Math. Phys. 314 (2012), no. 1, 193--228.
\bibitem{MS1} H. Matui and Y. Sato, 
\textit{ Strict comparison and $\Zjs$-absorption of nuclear $\mathrm{C}^*$-algebras},
Acta Math. 209 (2012), no. 1, 179--196.
\bibitem{MS2} H. Matui and Y. Sato, 
\textit{ $\mathcal{Z}$-stability of crossed products by strongly outer actions II}, preprint, 
arXiv:1205.1590. 
\bibitem{Ped} G. K. Pedersen,
\textit{ Unitary extensions and polar decompositions in a $\mathrm{C}^*$-algebra}, 
J. Operator Theory 17 (1987), no. 2, 357-364.
\bibitem{Pop1} S. Popa, 
\textit{ A short proof of ``injectivity implies hyperfiniteness'' for finite von Neumann algebras},
J. Operator Theory 16 (1986), no. 2, 261--272.
\bibitem{Pop2} S. Popa,
\textit{On local finite-dimensional approximation of $\mathrm{C}^*$-algebras},
Pacific J. Math. 181 (1997), no. 1, 141--158.
\bibitem{PS75CMP}
R. T. Powers and S. Sakai, 
\textit{Existence of ground states and KMS states 
for approximately inner dynamics}, 
Comm. Math. Phys. 39 (1975), 273--288. 

\bibitem{Ror0} M. R\o rdam, 
\textit{On the structure of simple $\mathrm{C}^*$-algebras tensored with a UHF-algebra},
 J. Funct. Anal. 100 (1991), no. 1, 1-17.
\bibitem{Ror1} M. R\o rdam, 
\textit{ The stable and the real rank of $\mathcal{Z}$-absorbing $\mathrm{C}^*$-algebras}, 
Internat. J. Math. 15 (2004), no. 10, 1065--1084. 
arXiv:math/0408020.
\bibitem{RW} M. R\o rdam and W. Winter, 
\textit{ The Jiang-Su algebra revisited}, 
J. Reine. Angew. Math. 642 (2010), 129--155. 
arXiv:0801.2259.
\bibitem{Sat1} Y. Sato,
\textit{Trace spaces of simple nuclear $\mathrm{C}^*$-algebras with finite-dimensional extreme boundary}, 
preprint, arXiv:1209.3000.  

\bibitem{S10JFA}
Y. Sato, 
\textit{The Rohlin property for automorphisms of the Jiang-Su algebra}, 
J. Funct. Anal. 259 (2010), 453--476. 
\bibitem{TWW} A. S. Toms, S. White, W. Winter,
\textit{ $\mathcal Z$-stability and finite dimensional tracial boundaries },
preprint, arXiv:1209.3292. 
\bibitem{Win07JFA}
W. Winter, 
\textit{Simple $\mathrm{C}^*$-algebras with locally finite decomposition rank}, 
J. Funct. Anal. 243 (2007), no. 2, 394--425. 
\bibitem{Win1} W. Winter, 
\textit{Decomposition rank and $\Zjs$-stability}
Invent. Math. 179 (2010), no. 2, 229-301.
\bibitem{WZ10Adv}
W. Winter and J. Zacharias, 
\textit{The nuclear dimension of $\mathrm{C}^*$-algebras}, 
Adv. Math. 224 (2010), no. 2,  461--498, 
arXiv:0903.4914.
\end{thebibliography}
\end{document}